\documentclass[a4paper,12pt,reqno]{amsart}
\usepackage[utf8]{inputenc}
\usepackage{amssymb}
\usepackage{amscd}
\usepackage{amsmath}
\usepackage{mathtools}
\usepackage{geometry}
\usepackage{scalerel}
\newcommand\Tau{\scalerel*{\tau}{T}}

\usepackage{amsthm,amsfonts}
\newtheorem{theorem}{Theorem}[section]

\newtheorem{corollary}{Corollary}[section]

\newtheorem{definition}{Definition}[section]

\newtheorem{lemma}{Lemma}[section]

\usepackage[x11names]{xcolor}
\newtagform{blue}{\color{RoyalBlue3}(}{)}
\newtagform{redandblue}[\textcolor{RoyalBlue3}]{\color{red}(}{)}
\usepackage{commath}
\usepackage[sc,osf]{mathpazo}
\usepackage[colorlinks=true, citecolor=blue, linkcolor=black, urlcolor=black]{hyperref}
\usepackage{hyperref}
\hypersetup{
  colorlinks=true,
 linkcolor=blue,
  filecolor=magenta,
urlcolor= cyan,
pdftitle={Sharelatex Example},
bookmarks=true,}
\usepackage{cleveref}
\usepackage{graphics,gensymb}
\usepackage[T1]{fontenc}
\usepackage{subfig}
\usepackage{epstopdf}
\usepackage[english]{babel}
\usepackage{array}
\usepackage{kpfonts}
\usepackage{booktabs}
\graphicspath{{./graphics/}}    

\title[]{\textbf{\lowercase{$k^{th}$} order Slant Hankel Operators on the Polydisk}}
\author[]{ M. P. Singh and Oinam Nilbir Singh,\\[1mm] Department of Mathematics,Manipur University, Canchipur,795003\\E\lowercase{mail: mpremjitmu@gmail.com}(M.P. S\lowercase{ingh) and\\ nilbirkhuman@manipuruniv.ac.in}(O.N. S\lowercase{ingh)}}

\begin{document}

\maketitle
\begin{abstract}
In this paper, we initiate the notion of $ k^{th} $ order  slant Hankel operators on $ L^2(\mathbb{T}^n) $ for $ k \geq 2$ and $ n \geq 1 $  where $ \mathbb{T}^n $ denotes the $ n $-torus. We give the necessary and sufficient condition for a bounded operator on $ L^2(\mathbb{T}^n) $ to be a $ k^{th} $ order slant Hankel and discuss their commutative, compactness, hyponormal and isometric property. 
\end{abstract}


\section{Introduction}
Throughout this paper, $ k $ is a fixed positive integer and $ k \geq 2 $. Suppose $\mathbb{D}$ is  the open unit disk  and $\mathbb{T}$ denotes the unit circle in the complex plane $\mathbb{C}$. For $n\geq 1$, let $\mathbb{D}^n$ and  $\mathbb {T}^n$  are respectively the polydisk in $\mathbb{C}^n$ and $n$-torus having a distinct boundary of $\mathbb{D}^n$. In this paper, $ z \in\mathbb{C}^n $ is used to denote $z = (z_1,z_2,\cdots ,z_n) $ and  $r =(r_1,r_2,\ldots,r_n) \;\; \forall r \in \mathbb{Z}^n $.
Let $ \phi(z) = \displaystyle \sum_{ r \in \mathbb{Z}^n} a_r z^r $, where $ a_r = a_{r_1,r_2,\cdots,r_n} $ and $ z^r = z_1^{r_1}\cdots \; z_n^{r_n} $  be a bounded
measurable function on the class of all essentially
bounded measurable functions on $ \mathbb{T}^n $,  $ L^\infty (\mathbb{T}^n)$. Then, the $i^{th}$ Fourier coefficient of $  \phi $ is given by $ \langle \phi, e_r(z) \rangle $ where $  e_r(z) = z^r $ is an orthonormal basis for $ {L}^2 (\mathbb{T}^n) $.
Let $ \varepsilon_{j} =(x_1,\cdots ,x_n) $ where $ {x}_\beta= \delta_{\beta \,j} $ jor $j = 1,2, \cdots,n$.
Let $ d \sigma$ be the Haar measure on $\mathbb{T}^n$.
The space $ {L}^2(\mathbb{T}^n)$ 
is defined as
\[ \displaystyle {L}^2(\mathbb{T}^n) = \bigg\lbrace \; {g} :\mathbb{T}^n \mapsto \mathbb{C} \;\, \vert \;\, {g(z)} =  \sum_{r \in \mathbb{Z}^n} {f}_r\,z^r,\;\sum_{r \in \mathbb{Z}^n} \vert \;{{f_r}} \;\vert^2 < \infty  \bigg\rbrace
 \]
where $ \mathbb{Z} $ 
denotes the set of integers.
The space $L^2(\mathbb{T}^n)$ is a Hilbert space with the norm given by the inner product $\displaystyle\langle g, \; h \rangle = \int_{\mathbb{T}^n} {g(z) \overline{h(z)}} d\sigma(z)$.

In the year $1911$, Toeplitz \cite{to} introduced the notion of Toeplitz operators. Hankel operators are the formal companions of Toeplitz operators. With the introduction of the class of Hankel operators in the middle of the
twentieth century, it has assumed tremendous importance due to its far reaching applications to problems of rational approximation, information and control theory, interpolation and prediction problems etc\cite{po}. 
 A finite matrix or a singly infinite matrix or a doubly infinite matrix  is a Hankel matrix  if its entries are constant along each skew diagonal. That is, the matrix $ (a_{m,m^{\prime}})$, for $ m $,$ m^{{\prime}}\in{\mathbb{Z}^n} $ is Hankel if $ a_{m_1,{m^{{\prime}}_1}}= a_{m_2,{m^{{\prime}}_2}}$
whenever ${m_1+{m^{{\prime}}_1}}= {m_2 +{m^{{\prime}}_2}}.$
The notion of Hankel operators have been generalized by many mathematicians. Avendano \cite{ra} initiated the notion
of $\lambda$-Hankel operators in the year 2000. Later in the year 2002, motivated by Barria
and Halmos's work, Avendano \cite{ra1} brought the concept of essentially
Hankel operators into the picture.  S.C. Arora \cite{ar} and his research associates introduced the class of slant Hankel operators in 2006. Datt and Porwal \cite{dp} introduced the concept on generalization of weighted slant Hankel operators in 2016. Hazarika and Marik \cite{ha} initiated the idea of Toeplitz and slant Toeplitz operators in the polydisk and discussed several properties in 2019. Recently, Datt and Pandey \cite{ds} discussed the Slant Toeplitz operators on the Lebesgue space of n-dimensional Torus.  For adequate literature on Toeplitz, slant Toeplitz, and the concepts of polydisk, one is referred to \cite{ru, Ru, ke}.
Motivated by the works of Ho \cite{ho, Ho}, in this paper,  we introduce the $ k^{th} $ order slant Hankel operator on $ L^2 (\mathbb{T}^n )$. We show in this paper that a bounded linear operator $  \mathcal{S}^{k,n} $ on $ {L}^2(\mathbb{T}^n) $ is a  $ k^{th} $ order slant Hankel operator of level $ n $ iff $  \mathcal{S}^{k,n} $ can be expressed as  $ k^{th} $ order slant Hankel matrix of level $ n $.
We also study  commutative, compactness, hyponormal and isometric property of the operator.  

\section{Preliminaries}

\begin{definition}\rm \cite{ds}
Let $ \phi \in L^{\infty}(\mathbb{T}^n)$ and $k$ be an integer $ \geq 2$. A $ k^{th} $ order slant Toeplitz operator $ A_\phi^{k,n} $ on the Lebesgue space $ L^{2}(\mathbb{T}^n)$ is defined as $  A_\phi^{k,n} = W_k^n M_{\phi} $, where $ M_{\phi} $ is the Laurent operator on  $ L^{2}(\mathbb{T}^n)$ induced by $ \phi $ and $ W_k^n $ is the linear operator on $ L^{2}(\mathbb{T}^n)$  given by 
\[
 W_k^n \; e_m =
\begin{cases}
 e_{\frac{m}{k}} & \mbox {if} \; k\mid m ; \\
 0 & {\rm otherwise},\; \forall \; m \; \in \; \mathbb{Z}^n.
 
 \end{cases}
\]

\end{definition}

Let $ \lbrace a_m \rbrace _{m \,\in\,\mathbb{Z}^n}$ be  a sequence of scalars. A matrix of the expression
\[ \mathcal{S}_{m_2,\cdots,m_n}^{k,1} = \begin{pmatrix}
\ddots & \vdots & \vdots & \vdots  & \vdots\\
\ldots & a_{(0,m_2,\cdots,m_n)} & a_{(-1,m_2,\cdots ,m_n)} & a_{(-2,m_2,\cdots,m_n)} & \cdots \\
\ldots & a_{(-2k,m_2,\cdots,m_n)} & a_{(-2k-1,m_2,\cdots,m_n)} & a_{(-2k-2,m_2,\cdots,m_n)} & \cdots \\
\ldots & a_{(-4k,m_2,\cdots,m_n)} & a_{(-4k-1,m_2,\cdots,m_n)} & a_{(-4k-2,m_2,\cdots,m_n)} & \cdots \\
\vdots & \vdots & \vdots & \vdots & \ddots
\end{pmatrix}\]
is called a $ k^{th} $ order slant Hankel matrix of level $ 1 $.

A block matrix  of the form
\[ \mathcal{S}_{m_3,...,m_n}^{k,2} = \begin{pmatrix}
\ddots & \vdots & \vdots & \vdots  & \vdots\\
\ldots & \mathcal{S}_{0,m_3,\cdots,m_n}^{1,k} & \mathcal{S}_{-1,m_3,\cdots,m_n}^{1,k } & \mathcal{S}_{-2,m_3,\cdots ,m_n}^{1, k} & \cdots \\
\ldots & \mathcal{S}_{-2k,m_3,\cdots,m_n}^{1,k} & \mathcal{S}_{-2k-1,m_3,\cdots,m_n}^{1,k} & \mathcal{S}_{-2k-2,m_3,\cdots,m_n}^{1,k} & \cdots \\
\ldots & \mathcal{S}_{-4k,m_3,\cdots,m_n}^{1,k} & \mathcal{S}_{-4k-1,m_3,\cdots,m_n}^{1,k} & \mathcal{S}_{-4k-2,m_3,...,m_n}^{1,k} & \cdots \\
\vdots & \vdots & \vdots & \vdots & \ddots
\end{pmatrix}\]
is called  a $ k^{th} $ order slant Hankel matrix of level $ 2 $. Following the pattern, the $ k^{th} $ order slant Hankel matrix of level $ n $ is define as   
\[ \mathcal{S}^{k,n} = \begin{pmatrix}
\ddots & \vdots & \vdots & \vdots  & \vdots\\
\ldots & \mathcal{S}_{0}^{k,n-1} & \mathcal{S}_{-1}^{k,n-1} & \mathcal{S}_{-2}^{k,n-1} & \cdots \\
\ldots & \mathcal{S}_{-2k}^{k,n-1} & \mathcal{S}_{-2k-1}^{k,n-1} & \mathcal{S}_{-2k-2}^{k,n-1} & \cdots \\
\ldots & \mathcal{S}_{-4k}^{k,n-1} & \mathcal{S}_{-4k-1}^{k,n-1} & \mathcal{S}_{-4k-2}^{k,n-1} & \cdots \\
\vdots & \vdots & \vdots & \vdots & \ddots
\end{pmatrix}.\]

\begin{definition}\label{mydefn7}
For  an integer $k$ $ \geq 2$, a slant Hankel operator $ \mathcal{S}_\phi^{k,n }$ of $ k^{th} $ order which is induced by $ \phi \in L^{\infty}(\mathbb{T}^n)$ on the Lebesgue space $ L^{2}(\mathbb{T}^n)$ is given by $ \mathcal{S}_\phi^{k,n } = \mathcal{V}_k^n M_{\phi} $, where $ M_{\phi} $ is the Multiplication operator on  $ L^{2}(\mathbb{T}^n)$ and $\mathcal{V}_k^n $  is a linear operator defined on $ {L}^2(\mathbb{T}^n)$ by 
\[
 \mathcal{V}_k^n\; e_m =
\begin{cases}
 e_{\frac{-m}{k}} & \mbox {if} \; k \mid m ; \\
 0 & { \rm otherwise},\forall \; m \; \in \; \mathbb{Z}^n.
 
 \end{cases}
\]
\end{definition}
Let $ \displaystyle h(z) = \displaystyle \sum_{ m \in\mathbb{Z}^n } u_m  z^m  \in{L}^2(\mathbb{T}^n) $, then  $ \mathcal{V}_k^n \, h(z) = \displaystyle\sum_{ m\in\mathbb{Z}^n } u_{km}  z^{-m} $. 
Further, the adjoint of $ \mathcal{V}_k^n $ is given by  $ (\mathcal{V}_k^n)^{\ast} \; e_m = e_{-km} $ for each $ m\in\mathbb{Z}^n $. 
So, $ (\mathcal{V}_k^n)^{\ast} \, h(z) =  \displaystyle\sum_ { k \in\mathbb{Z}^n } u_{k}  z^{-km}$. Then $ \mathcal{V}_k^n \,(\mathcal{V}_k^n)^\ast = I $ and $ (\mathcal{V}_k^n)^\ast\,\mathcal{V}_k^n = {P}_e $, where $ {P}_e $ is the projection on  the closed  span of $ \lbrace e_{km} \;\; : \; m \in\mathbb{Z}^n \rbrace $ in $ {L}^2(\mathbb{T}^n) $ . Thus $ \mathcal{V}_k^n $ is the  co-isometry on $ L^2(\mathbb{T}^n) $ and isometry on $ {P}_e({L}^2(\mathbb{T}^n)) $.
In fact, $ \Vert \mathcal{V}_k^n  \Vert = 1 $. Let $\mathcal{B}(L^2(\mathbb{T}^n))$ be the algebra of all bounded linear operators on $ L^2(\mathbb{T}^n) $ and $ \Omega :L^\infty(\mathbb{T}^n)\mapsto \mathcal{B}(L^2(\mathbb{T}^n))$  be defined as $ \Omega =\mathcal{S}_\phi^{k,n} $ then $  \Omega $  is linear and $1-1$.

\begin{lemma}\label{lemma01}
A bounded linear operator $S$ on $ L^2(\mathbb{T}^n) $ is a $ k^{th} $ order slant Hankel operator iff  $ M_{z_j}S = SM_{z_j^{-k}} $.
\end{lemma}

\begin{lemma}\label{lemma1}
$ M_{z_j}S = SM_{z_j^{-k}} $ iff $ M_{z^m} S = S M_{z^{-km}} \; \forall 1 \leq j\leq n \; and\;  m \in \mathbb{Z}^n$ for  a bounded linear operator $ S $ on $ L^2(\mathbb{T}^n) $.
\end{lemma}
\begin{proof}
Let $ M_{z^m} S = S M_{z^{-km}} \forall m \in \mathbb{Z}^n$. Taking $ m = \epsilon_j \;{\rm for} 1 \leq j\leq n  $ implies $ M_{z_j}S = SM_{z_j^{-k}} $. For the converse  part consider $ m = (m_1, m_2, \cdots, m_n)$, then \[ M_{z^m} S = S M_{z_1^{m_1}, \cdots, z_n^{m_n}} S  = S M_{z^{-km}} \].
\end{proof} 
\begin{lemma}\label{lemma2}
  A bounded linear operator $\mathcal{S}_\phi^{k,n } $ is a $ k ^{th}$  order slant Hankel operator of level $n$ on $ L^2(\mathbb{T}^n) $  if and only if $ M_{z^m}\mathcal{S}_\phi^{k,n } = \mathcal{S}_\phi^{k,n } M_{z^{-km}} \forall m \in \mathbb{Z}^n$. 
\end{lemma}
\begin{theorem}\label{thm1}
If $ \mathcal{S}_\phi^{k,n } $  is a $ k ^{th}$  order slant Hankel operator on $ L^2(\mathbb{T}^n) $ then $ M_\phi \mathcal{S}_\psi^{k,n }$ is a  $ k ^{th}$  order slant Hankel operator and $  M_\phi \mathcal{S}_\psi^{k,n } = \mathcal{S}_{\phi(z^{-k}) \psi_k(z)}^{k,n} .$ 
\end{theorem}
\begin{proof}
$ M_{z^m} (M_\phi\mathcal{S}_\psi^{k,n})= M_\phi   M_{z^m}\mathcal{S}_\psi^{k,n} = (M_\phi \mathcal{S}_\psi^{k,n})M_{z^{-km}} \;\forall m \in \mathbb{Z}^n.$ Therefore by Lemma\;{\ref{lemma2}} $ M_\phi\mathcal{S}_\psi^{k,n} $ is a $ k ^{th}$  order slant Hankel operator. $ \mathcal{S}_\psi^{k,n} $ is a  $ k ^{th}$  order slant Hankel operator on $ L^2(\mathbb{T}^n) $ implies $ M_{z^m}\mathcal{S}_\psi^{k,n} = \mathcal{S}_\psi^{k,n} M_{z^{-km}} \;\forall m \in \mathbb{Z}^n $. This in turn implies $ M_{\phi(z)}\mathcal{S}_\psi^{k,n} = \mathcal{S}_\psi{k,n} M_{\phi(z^{-k})} \;\forall \phi \in L^\infty(\mathbb{T}^n).$ Now, $ M_\phi \mathcal{S}_\psi^{k,n} = \mathcal{S}_\psi^{k,n} M_{\phi(z^{-k})} = \mathcal{V}_k^n M_\psi M_{\phi(z^{-k})} = \mathcal{V}_k^n M_{\psi(z) \phi(z^{-k})} = \mathcal{S}_{\psi(z) \phi(z^{-k})}^{k,n }.$
\end{proof}
\begin{theorem}\label{thm2}
$  \mathcal{S}_\phi^{k,n} M_\psi = M_\psi \mathcal{S}_\phi^{k,n} $ iff   $ \phi(z) \psi(z) = \phi(z) \psi(z^{-k})\; \forall  \; z \in \mathbb{T}^n $ and $ \psi(z)= constant $ iff $  \ \mathcal{S}_\phi^{k,n} M_\psi = M_\psi \mathcal{S}_\phi^{k,n} $ for a particular $ \phi $ which is invertible.
\end{theorem}
\begin{proof}  By Theorem \ref{thm2},
\begin{equation}\label{eqn1}
\mathcal{S}_\phi^{k,n} M_\psi = \mathcal{V}_k^n M_\phi M_\psi = \mathcal{V}_k^n M_{\phi\psi} = \mathcal{S}_{\phi\psi}^{k,n} {\rm and} \; M_\psi \mathcal{S}_\phi^{k,n}  = \mathcal{S}_{\psi(z^{-k}) \phi(z)}^{k,n} 
\end{equation}
Using equation (\ref{eqn1})  we can say $  \mathcal{S}_\phi^{k,n} M_\psi = M_\psi\mathcal{S}_\phi^{k,n} \Leftrightarrow \mathcal{S}_{\phi(z) \psi(z)}^{k,n} = \mathcal{S}_{\psi(z^{-k})\phi(z) }^{k,n} \Leftrightarrow \phi(z) \psi(z) = \phi(z) \psi(z^{-k}) .$  
If $ \psi(z) = constant $ then, $ \psi(z) = \psi(z^{-k}) \Rightarrow \mathcal{S}_\phi^{k,n} M_\psi = M_\psi\mathcal{S}_\phi^{k,n}.$

For converse part, let $ \phi $ be invertible. Then $ \exists \; \phi^{-1} \in L^\infty (\mathbb{T}^n)$ such that $ \phi \phi^{-1} = I = \phi^{-1} \phi $. So, $ \phi(z) \psi(z) = \phi(z)\psi(z^{-k}) \Leftrightarrow \psi(z) = \psi(z^{-k})$. $ \psi(z) = \displaystyle \sum_{m \in \mathbb{Z}^n} a_m z^m $ implies $\psi(z^{-k}) = \displaystyle\sum_{m \in \mathbb{Z}^n} a_m z^{-k m} $. Now, $ \psi(z)= \psi(z^{-k}) \Rightarrow \displaystyle \sum_{m \in \mathbb{Z}^n} a_m z^m = \displaystyle\sum_{m \in \mathbb{Z}^n} a_m z^{-k m} \Rightarrow \displaystyle\sum_{0 \neq m \in \mathbb{Z}^n,k \nmid m }a_m z^m + \sum_{0 \neq m \in \mathbb{Z}^n}( a_{-km}- a_m)z^{-km} = 0 $. Therefore, $ a_m = 0 \; \forall k \nmid m $ and $ a_{-km}= a_m \; \forall 0 \neq m \in \mathbb{Z}^n $. Hence, $ \psi(z) = a_0 $, which is a constant.
\end{proof}
\begin{theorem}\label{thm01}
 $ \mathcal{S}_\phi^{k,n} $ is a $ k^{th} $ order slant Hankel operator on $ L ^2(\mathbb{T}^n)$ iff $ \big\langle \mathcal{S}_\phi^{k,n}  e_{m-k{\varepsilon_j}},e_{m^{\prime} + {\varepsilon_j}} \big\rangle = \big\langle \mathcal{S}_\phi^{k,n}  e_m , e_{m^{\prime}} \big\rangle $ $ \forall \; m , m^{\prime} \in\mathbb{Z}^n , 1 \leq j \leq n.$
\end{theorem}
\begin{proof}
 Let $\lbrace a_{\eta , \, \zeta} \rbrace_{\eta ,\, \zeta \;\in \;\mathbb{Z}_+^n}$ be scalars such that $\displaystyle \Tau e_{\zeta} =  \sum_{\eta \in\mathbb{Z}_+^n} a_{\eta,\,\zeta} \,e_{\eta} \quad  \forall \; \zeta\in\mathbb{Z}_+^n $. Since, $ \mathcal{G}= \{e_m \}_{\,m\,\in\,\mathbb{Z}^n}$ is an orthonormal basis for $ {L}^2 (\mathbb{T}^n) $, so for arbitrarily fixed  $ {m_j, m_{j +1} \cdots, m_n }\in\mathbb{Z} $ and if $ \mathcal{G}_{[m_j,\cdots ,m_n]}= \{ e_{(m_1,\cdots ,m_n)}: {m_\beta} \, \in\mathbb{Z} $ for $ 1\leq \; \beta  < \; j \}$, then $ \mathcal{G}_{[m_j,\cdots ,m_n]}$ is an orthonormal basis for $ {L}^2(\mathbb{t}^{\;j -1}) $. For $ m =(m_1,m_2,\cdots ,m_n) $, $ m^{\prime}=( m^{{\prime}}_1, m^{{\prime}}_2,\cdots , m^{{\prime}}_n) \in\mathbb{Z}^n$ let $ \big\langle \, \Tau \, e_{m - k\varepsilon_j} , e_{m^{{\prime}} +\varepsilon_j } \big\rangle = \big\langle \, \Tau \, e_m  , e_{m^{{\prime}}}   \big\rangle \, \forall \; m, \, {m^{\prime} \in\mathbb{Z}^n} $ and $ 1 \leq \; j \leq  \; n $. Vary $ m_1 $, $ m^{{\prime}}_1$ and fix $(m_2,\cdots ,m_n)$, $(m^{{\prime}}_2,\cdots ,m^{{\prime}}_n)$ 
\begin{align*}
&\quad \;\; \bigg\langle \, \Tau \, e_{m - k\varepsilon_1}(z) , e_{ m^{\prime} +\varepsilon_1 } (z) \bigg\rangle = \bigg\langle \, \Tau \, e_m(z)  , e_{ m^{{\prime}}}  (z) \bigg\rangle \cr 
&\displaystyle\Rightarrow  \sum_{\eta \in\mathbb{Z}^n} a_{\eta,\, m -k\varepsilon_1} \big \langle e_{\eta}(z),e_{  m^{{\prime}} + \varepsilon_1}(z)\big \rangle = \sum_{\eta \in\mathbb{Z}^n} a_{\eta,\, m} \big \langle e_{\eta}(z),e_{ m^{{\prime}}}(z)\big \rangle \cr 
&\Rightarrow  a_{ m^{{\prime}}+\varepsilon_1,m -k\varepsilon_1} = \; a_{ m^{{\prime}},m} \quad \forall\, m, m^{\prime} \in\mathbb{Z} .
\end{align*}
So, $ \Tau : \mathcal{G}_{1,[m_2,\cdots ,m_n]} \mapsto \mathcal{G}_{1,[ m^{{\prime}}_2,\cdots , m^{{\prime}}_n]} $ is a $k^{th}$ order slant Hankel matrix of level $ 1 $, $ \mathcal{S}^{k,1}_{( m^{{\prime}}_2,\cdots , m^{{\prime}}_n)(m_2,\cdots ,m_n)} $. Varying $ m_2 $,$ m^{{\prime}}_2 $ and fixing $(m_3,\cdots ,m_n)$ and $ (m^{{\prime}}_3,\cdots ,m^{{\prime}}_n) $,
\begin{align*}
&\quad \;\;  \Big\langle \, \Tau \, e_{m - k\varepsilon_2}, e_{m^{{\prime}} +\varepsilon_2 }  \Big\rangle = \Big\langle \, \Tau \, e_m  , e_m^{{\prime}}  \Big\rangle \cr
& \Rightarrow \mathcal{S}_{(m^{{\prime}}_2+ 1,m^{{\prime}}_3,\cdots,m^{{\prime}}_n)(m_2-k,m_3,\cdots ,m_n)}^{k,1} = \mathcal{S}_{(m^{{\prime}}_2,\cdots ,m^{{\prime}}_n)(m_2,\cdots ,m_n)}^{k,1} .
\end{align*}
Then, $ \Tau : \mathcal{G}_{2,[m_3,\cdots ,m_n]} \mapsto \mathcal{G}_{2,[ m^{{\prime}}_3,\cdots,m^{{\prime}}_n]} $  can be expressed as  $k^{th}$ order slant Hankel  matrix of level $ 2 $, $ \mathcal{S}_{(m^{{\prime}}_3,\cdots,m^{{\prime}}_n)(m_3,\cdots,m_n)}^{(k,2)} $.
Further, varying $ m_3 $, $ m^{{\prime}}_3 $ and fixing  $(m_4,\cdots,m_n)$, $ (m^{{\prime}}_4,\cdots,m^{{\prime}}_n) $
\begin{align*} 
&\quad \;\; \Big\langle \, \Tau \, e_{m - k\varepsilon_3}, e_{m^{{\prime}} +\varepsilon_3 }  \Big\rangle = \Big\langle \, \Tau \, e_m  , e_{m^{\prime}}  \Big\rangle \cr
&\Rightarrow \mathcal{S}_{(m_3+ 1,m_4,\cdots,m_n)(m^{{\prime}}_3-k,m^{{\prime}}_4,\cdots,m^{{\prime}}_n)}^{(k,2)} = \mathcal{S}_{(m^{{\prime}}_3,\cdots,m^{{\prime}}_n)(m_3,\cdots,m_n)}^{(2)} .
\end{align*}
Hence, $ \Tau : \mathcal{G}_{3,[m_4,\cdots,m_n]} \mapsto \mathcal{G}_{3,[m^{{\prime}}_4,\cdots,m^{{\prime}}_n]} $  can be expressed as $k^{th}$ order slant Hankel  matrix of level $ 3 $, $ \mathcal{H}_{(m^{{\prime}}_4,\cdots,m^{{\prime}}_n)(m_4,\cdots,m_n)}^{(3)} $.
Similarly, continuing the same process, it  can be concluded that $ \Tau : {L}^2(\mathbb{T}^n)\mapsto {L}^2(\mathbb{T}^n) $ can be expressed as $k^{th}$ order slant Hankel matrix of level $ n $ after $ n $ steps.
Conversely, suppose $ \Tau $ can be expressed as $k^{th}$ order slant Hankel matrix of level $ n $. So, for $ m = (m_1,\cdots,m_n) $ 
and $ m^{{\prime}} = (m^{{\prime}}_1,\cdots,m^{{\prime}}_n) $, 
$ \Tau : \mathcal{G}_{j-1,[m_j,\cdots,m_n]} \mapsto \mathcal{G}_{j-1,[m^{{\prime}}_j,\cdots,m^{{\prime}}_n]} $ can be represented as a $k^{th}$ order slant Hankel matrix of  level $ (j-1) $,  $ \mathcal{S}_{(m^{{\prime}}_j,\cdots,m^{{\prime}}_n)(m_j,\cdots,m_n)}^{k,(j-1)} $ for $ j = 2, \cdots ,n $. Therefore,
$ \Big\langle \, \Tau \, e_{m - k\varepsilon_j}, e_{m^{{\prime}} +\varepsilon_j }  \Big\rangle = \Big\langle \, \Tau \, e_m  , e_{m^{\prime}}  \Big\rangle \; \forall \;  j = 1,2,\cdots,n-1.$
Lastly, let's consider the $ n-1 $ torus $ {L}^2(\mathbb{T}^{n-1} )$, then ${L}^2(\mathbb{T}^{n-1} )$ is an isomorphic copy of   $\mathcal{G}_{n-1,(m_n)}$ for each $ m_n \in \mathbb{Z}$. Hence, $ n $ torus $ {L}^2(\mathbb{T}^{n} )$ can be expressed as $ {L}^2(\mathbb{T}^{n}) = \displaystyle\oplus _{m_n \in\mathbb{Z}} \; \mathcal{G}_{n-1,(m_n)}$. So, $ \Tau : {L}^2(\mathbb{T}^n)\mapsto {L}^2(\mathbb{T}^n) $ is a  $ k^{th} $ order slant Hankel matrix of level n where $ (m^{{\prime}}_n,m_n)^{th}$ entry is $ \mathcal{S}_{m^{{\prime}}_n,m_n}^{(k,n-1)} $ and also we have $ \mathcal{S}_{m^{{\prime}}_n+1,m_n-k}^{k,(n-1)} =  \mathcal{S}_{m^{{\prime}}_n,m_n}^{(n-1)} $. Hence, $ \big\langle \, \Tau \, e_{m - k\varepsilon_n}, e_{m^{{\prime}} +\varepsilon_n } \big\rangle =\big\langle \, \Tau \, e_m  , e_{m^{\prime}}  \big\rangle $. 
\end{proof}

\begin{lemma}\label{lemma3}
If $ \mathcal{S}_\phi^{k,n} $ is a $ k^{th} $ order slant Hankel operator on $ L ^2(\mathbb{T}^n)$ then $ \big\langle \mathcal{S}_\phi^{k,n}  e_{m-k{\varepsilon_j}},e_{m^{\prime} + {\varepsilon_j}} \big\rangle = \big\langle \mathcal{S}_\phi^{k,n}  e_m , e_{m^{\prime}} \big\rangle $ $ \forall \; m , m^{\prime} \in\mathbb{Z}^n , 1 \leq j \leq n.$
\end{lemma}
\begin{proof} We have
\begin{eqnarray*}
\big\langle \mathcal{S}_\phi^{k,n} e_{m-k{\varepsilon_j}}(z),e_{m^{\prime} + {\varepsilon_j}}(z) \big\rangle
&= &\big\langle  \phi(z) z^{m-k\varepsilon_j} ,(\mathcal{V}_k^n)^{\ast} ( z^{ m^{\prime} +{\varepsilon_j}}) \big\rangle \\
&= & \displaystyle\sum_{ r \in\mathbb{Z}^n } a_r \bigg\langle z^{r + m -k\varepsilon_j},z^{-k m^{\prime}-k\varepsilon_j} \bigg\rangle \\
&= &\displaystyle\sum_ { r \in\mathbb{Z}^n } a_r \bigg\langle z^{r+m} , z^{-km^{\prime}} \bigg\rangle \\
&=&  \big\langle \phi(z) z^m , (\mathcal{V}_k^n)^{\ast} (z^{m^{\prime}})\big\rangle \\
&=& \big\langle \mathcal{S}_{\phi}^{k,n} e_m(z),e_{m^{\prime}}(z) \big\rangle .
\end{eqnarray*}
\end{proof}

\begin{theorem}\label{mythm8}
A bounded linear operator $ \mathcal{S} $ on $ {L}^2(\mathbb{T}^n) $ is a $k^{th}$ order slant Hankel operator of level $ n $ iff $ \mathcal{S} $ can be represented as a $k^{th}$ order slant Hankel matrix of level $ n $.
\end{theorem}
\begin{proof}
Let $ \mathcal{S} $ be a $k^{th}$ order slant Hankel operator of level $ n $ on $ {L}^2(\mathbb{T}^n) $.
Then, by Definition \ref{mydefn7},  $ \mathcal{S} = \mathcal{S}_{\phi}^{k,n} $ for some $ \phi \in\it{L}^\infty(\mathbb{T}^n) $.
If $(\beta_{m,m^{\prime}})_{m,m^{\prime} \in\;\mathbb{Z}^n}$ is the matrix representation of $ \mathcal{S}_{\phi}^{k,n} $ with respect to the  orthonormal basis, then by Lemma \ref{lemma3}
\[ \big\langle \mathcal{S}_{\phi} e_{m-k{\varepsilon_j}},e_{m^{\prime} + {\varepsilon_j}} \big\rangle = \big\langle \mathcal{S}_{\phi} e_m , e_{m^{\prime}} \big\rangle \quad \forall \; m , m^{\prime}  \in \; \mathbb{Z}^n , 1 \leq j \leq n. \]
Therefore, $(\beta_{m,m^{\prime}})_{m,m^{\prime} \, \in\,\mathbb{Z}^n}$  is a $k^{th}$ order slant Hankel matrix of level $ n $.\\
Conversely, let the matrix $(\beta_{m,m^{\prime}})_{ m ,m^{\prime} \in\;\mathbb{Z}^n}$ of $ \mathcal{S} $ be a slant Hankel matrix of level $ n $. Then,
\begin{align}\label{myeqn8}
\big\langle \mathcal{S} e_m , e_{m^{\prime}} \big\rangle =  (\beta_{m,m^{\prime}}) =  (\beta_{m^{\prime}+ \varepsilon_j,m - k\varepsilon_j})  =  \big\langle \mathcal{S} e_{m -k\varepsilon_j }, e_{m^{\prime} + \varepsilon_j} \big\rangle.
\end{align}
Now,
\begin{eqnarray}\label{myeqn9}
\big\langle {M}_{z_j} \mathcal{S} e_m , e_{m^{\prime}} \big\rangle & = &\big\langle \mathcal{S} e_m , {M}_{\overline z_j} e_{m^{\prime}} \big\rangle  =  \big\langle \mathcal{S} e_m , e_{m^{\prime} - \varepsilon_j} \big\rangle  \nonumber\\
&= &\big\langle \mathcal{S} e_{m - k\varepsilon_j} , e_{m^{\prime}} \big\rangle  =  \big\langle \mathcal{S} {M}_{z_j^{-k}}e_m , e_{m^{\prime}} \big\rangle 
\end{eqnarray}
\begin{align}\label{myeqn10}
&\Rightarrow {M}_{z_j} \mathcal{S} e_m = \mathcal{S} {M}_{z_j^{-k}}e_m  \quad \forall\;\; m \in\mathbb{Z}^n \cr
& \Rightarrow {M}_{z_j} \mathcal{S}  = \mathcal{S} {M}_{z_j^{-k}}.
\end{align}
Hence, by Lemma \ref{lemma01},  \; $ \mathcal{S} $ is a slant Hankel operator.
\end{proof}

\section{Properties} 
\begin{theorem}\label{mythm9}
For all $ m \in\mathbb{Z}^n $, $ \mathcal{V}_k^n \, M_{z^{-km}} \, \mathcal({V}_k^n)^{\ast} = {M}_{z^m}$  and $\mathcal{V}_k^n \, M_{z^{-l}} \, \mathcal({V}_k^n)^{\ast} = 0 $ if $ k\nmid l, l \in\mathbb{Z}^n $.
\end{theorem}
\begin{theorem}\label{thm3}
$ \mathcal{V}_k^n \mathcal{S}_{\phi}^{k,n} $ is a $ k^{th} $ order slant Hankel operator of level $n$ iff $ \phi = 0 $. 
\end{theorem}
\begin{proof}
Suppose, $ \mathcal{V}_k^n \mathcal{S}_{\phi}^{k,n} $ is a $ k^{th} $ order slant Hankel operator of level $n$. Then, for all $ m, m^{\prime} \in \mathbb{Z}^n,1 \leq j \leq n,$
\begin{eqnarray}\label{eqn2}
&& \big\langle \mathcal{V}_k^n \mathcal{S}_{\phi}^{k,n}e_{m-k{\varepsilon_j}},e_{m^{\prime} + {\varepsilon_j}}  \big\rangle = \big\langle \mathcal{V}_k^n \mathcal{S}_{\phi}^{k,n}e_m, e_m^{\prime} \big\rangle  \nonumber\\
&\Rightarrow & \bigg\langle \mathcal{V}_k^n \displaystyle \sum_{r \in \mathbb{Z}^n} a_{-kr - m + k \varepsilon_j}z^r,z^{m^{\prime} + \varepsilon_j}\bigg\rangle = \bigg\langle \mathcal{V}_k^n \displaystyle \sum_{r \in \mathbb{Z}^n} a_{-kr - m} z^r,z^{m^{\prime}} \bigg\rangle  \nonumber\\
&\Rightarrow &\bigg\langle \sum_{r \in \mathbb{Z}^n} a_{-k^2r-m+k\varepsilon_j}z^{-r}, z^{m^{\prime}+\varepsilon_j} \bigg\rangle = \bigg\langle \sum_{r \in \mathbb{Z}^n} a_{k^2r-m} z^r,z^{m^{\prime}}\bigg\rangle \nonumber\\
&\Rightarrow & a_{k^2 m^{\prime}- m + k^2 \varepsilon_j + k \varepsilon_j }= a_{k^2m^{\prime} -m }
\end{eqnarray}
Now, for each $j = 1, 2, \cdots ,n $ and $ k \in \mathbb{Z}^n $, $a_{-m} = a_{-m +(k^2 + k ) \varepsilon_j} = a_{-m +2(k^2 + k ) \varepsilon_j} = \cdots $. Here,  $ \Delta(k^2 + k )\varepsilon_j \rightarrow \infty $ as $ \Delta \rightarrow \infty $ and $ \phi \in L^\infty(\mathbb{T}^n) $ so, $  a_{-m +  \Delta(k^2 + k ) \varepsilon_j} \rightarrow 0 $ as $ \Delta \rightarrow \infty $. Therefore, $ a_m = 0 \;\; \forall m \in \mathbb{Z}^n  $ which implies $\phi = 0$.
\end{proof}
\begin{theorem}\label{thm4}
For $ \phi , \psi \in L^\infty(\mathbb{T}^n)$, the following are equivalent
\begin{enumerate}
\item[(a)] $\mathcal{S}_{\phi}^{k,n} \mathcal{S}_{\psi}^{k,n}  $ is a $k^{th}$ order slant Hankel operator of level $n$ 
\item[(b)] $ \phi(z^{-k}) \psi(z) = 0 $
\item[(c)] $\mathcal{S}_{\phi}^{k,n} \mathcal{S}_{\psi}^{k,n} = 0 $
\end{enumerate}
\end{theorem}
\begin{proof}
First, we prove $ (a)\Leftrightarrow (b)$. Here, $ \mathcal{S}_{\phi}^{k,n} \mathcal{S}_{\psi}^{k,n} =  \mathcal{V}_k^n M_\phi \mathcal{S}_{\psi}^{k,n} = \mathcal{V}_k^n \mathcal{S}_{\phi(z^{-k})\psi(z)}^{k,n} $ (by Theorem \ref{thm1}) and by Theorem \ref{thm3}  $, \mathcal{S}_{\phi}^{k,n} \mathcal{S}_{\psi}^{k,n}  $ is a $k^{th}$ order slant Hankel operator of level $n$ if and only if $ \phi(z^{-k})\psi(z) = 0 $.

Next, we prove $ (b)\Leftrightarrow (c) $. Now, $ \mathcal{S}_{\phi}^{k,n} \mathcal{S}_{\psi}^{k,n} = 0 \Leftrightarrow  \mathcal{V}_k^n \mathcal{S}_{\phi(z^{-k})\psi(z)}^{k,n} = 0 \Leftrightarrow \mathcal{V}_k^n \mathcal{S}_{\phi(z^{-k})\psi(z)}^{k,n} $  is a $k^{th}$ order slant Hankel operator of level $n$ $ \Leftrightarrow \phi(z^{-k})\psi(z)= 0 $.
\end{proof}
\begin{corollary}
$ (\mathcal{S}_{\phi}^{k,n})^2 = \mathcal{S}_{\phi}^{k,n} $ if and only if $\phi = 0 $.
\end{corollary}
\begin{theorem}\label{thm5}
$ (\mathcal{S}_{\phi}^{k,n})^* $ is a $k^{th}$ order slant Hankel operator of level $n$  if and only if $\phi = 0 $. 
\end{theorem}
\begin{proof}
If $ \phi = 0 $, then $ \mathcal{S}_{\phi}^{k,n} = 0 = (\mathcal{S}_{\phi}^{k,n})^*  $. Let  $ (\mathcal{S}_{\phi}^{k,n})^* $  be a $k^{th}$ order slant Hankel operator of level $n$ . Then, $ \forall m, m^{\prime} \in \mathbb{Z}^n $
\begin{eqnarray}\label{eqn3}
&& \big\langle (\mathcal{S}_{\phi}^{k,n})^* e_{m-k{\varepsilon_j}},e_{m^{\prime} + {\varepsilon_j}} \big\rangle = \big\langle (\mathcal{S}_{\phi}^{k,n})^* e_m,e_{m^{\prime}} \big\rangle \nonumber\\
&\Rightarrow & \bigg\langle z^{m - k \varepsilon_j }, \mathcal{V}_k^n \displaystyle  \sum_{t \in \mathbb{Z}^n} a_t z^{t + m^\prime + \varepsilon_j}\bigg \rangle = \bigg \langle z^m, \mathcal{V}_k^n   \sum_{t \in \mathbb{Z}^n} a_t z^{t + m^\prime }\bigg \rangle \nonumber\\
&\Rightarrow &  \bigg \langle z^{m-k\varepsilon_j},  \sum_{t \in \mathbb{Z}^n} a_{-kt - m^\prime - \varepsilon_j} z^t  \bigg \rangle = \bigg \langle z^m, \sum_{t \in \mathbb{Z}^n} a_{-kt - m^\prime} z^t \bigg \rangle \nonumber\\
&\Rightarrow & a_{-km - m^\prime + (k^2 - 1 )\varepsilon_j} = a_{-mk - m^\prime } \;\; \forall m, m^\prime \in \mathbb{Z}^n \nonumber\\
&\Rightarrow &  a_{-km + (k^2 - 1)\varepsilon_j} = a_{-mk}.
\end{eqnarray}
 The proof  immediately follows from Theorem \ref{thm3} and hence  we obtain $ \phi =  0 $.
\end{proof}
\begin{theorem}\label{thm6}
The $k^{th} $ order slant Hankel operator of level $n$  $ \mathcal{S}_{\phi}^{k,n} $ is hyponormal if and only if $\phi = 0$.
\end{theorem} 
\begin{proof}
Let $ f(z) = e_m(z) $ and $\phi(z) = \displaystyle\sum_{r \in \mathbb{Z}^n} a_r z^r $. Let  $ \mathcal{S}_{\phi}^{k,n} $ be a hyponormal operator then for $ f \in L^2(\mathbb{T}^n)$,
\begin{equation}\label{myeqn4}
\big\Vert \mathcal{S}_{\phi}^{k,n}  f(z)\big\Vert \geq \big \Vert (\mathcal{S}_{\phi}^{k,n})^* f(z) \big \Vert \Rightarrow \big\Vert \mathcal{V}_k^n \phi(z)\big\Vert \geq \big\Vert \overline{\phi(z)}(\mathcal{V}_k^n)^* f(z) \big\Vert
\end{equation}
\begin{enumerate}
\item[Case 1.]\label{case1} $k \vert m $. {\rm From equation (\ref{myeqn4}) we get},
\begin{eqnarray}
&& \bigg \Vert \mathcal{V}_k^n) \displaystyle \sum_{r \in \mathbb{Z}^n} a_r z^{r + m } \bigg \Vert \geq \bigg \Vert \sum_{r \in \mathbb{Z}^n} \overline{a}_r z^{-r -km}\bigg \Vert \nonumber\\
&\Rightarrow &\bigg \Vert \sum_{r \in \mathbb{Z}^n} a_{-kr -m} z^r \bigg \Vert^2 \geq \bigg \Vert \sum_{r \in \mathbb{Z}^n} \overline{a}_ {-r - km } z^r \bigg \Vert^2 \nonumber\\
&\Rightarrow & \sum_{r \in \mathbb{Z}^n} \vert a_{-kr -m} \vert^2 \geq \sum_{r \in \mathbb{Z}^n} \vert a_{-r - km } \vert^2 \;\; \forall m \in \mathbb{Z}^n \nonumber\\
&\Rightarrow & \sum_{r \in \mathbb{Z}^n} \vert a_{kr} \vert^2 \geq \sum_{r \in \mathbb{Z}^n} \vert a_r \vert^2 \nonumber\\
&\Rightarrow &\sum_{r \in \mathbb{Z}^n, k \nmid r}  \vert a_r \vert^2 \leq 0 \nonumber\\
&\Rightarrow & \; a_r = 0 \;\; \forall r \in \mathbb{Z}^n, k \nmid r.
\end{eqnarray}
\item[Case 2.]\label{case2} $k \nmid m$. {\rm By equation (\ref{myeqn4})},
\begin{flalign} \label{eqn6}
0 \geq  \bigg \Vert \displaystyle \sum_{r \in \mathbb{Z}^n} \overline{a}_r z^{-(r + km)} \bigg \Vert^2 \Rightarrow \sum_{r \in \mathbb{Z}^n} \vert a_{r - km }\vert^2 \leq 0.
\end{flalign}
\rm If $k \nmid r $, then $ a_r = 0 \;\; \forall r \in \mathbb{Z}^n $ ( by case \ref{case1}). Again, if $ k \mid r $, then by equation (\ref{eqn6})  $\displaystyle \sum_{r \in \mathbb{Z}^n} \vert a_r \vert^2 \leq 0  \Rightarrow a_r = 0 \;\; \forall r \in \mathbb{Z}^n $.
\end{enumerate}
Hence, by case [\ref{case1}] and case [\ref{case2}], we obtain $ \phi = 0$
\end{proof}
\begin{theorem}\label{thm7}
The zero operator is the only compact $ k^{th} $ order slant Hankel operator of level $n$, on $ L^2(\mathbb{T}^n)$.
\end{theorem}
\begin{proof}
Let  $ \mathcal{S}_{\phi}^{k,n} $ be a compact operator. Since the ideal space of bounded linear operators on the Hilbert space are compact, therefore $ \mathcal{S}_{\phi}^{k,n} M_{z^t}$ is compact where $ t = 1, 2, \cdots ,k^2-1 $ implies $ (\mathcal{S}_{\phi}^{k,n} M_{z^t})^*$ is compact. Now, $ (\mathcal{S}_{\phi}^{k,n} M_{z^t})^* = (\mathcal{V}_k^n M_{z^t \phi})^* =( \mathcal{S}_{z^t\phi}^{k,n})^*$. So, $ \mathcal{V}_k^n (\mathcal{S}_{\phi}^{k,n} M_{z^t})^* e_q  =  \mathcal{V}_k^n ( \mathcal{S}_{z^t\phi}^{k,n})^* e_q =   \mathcal{V}_k^n (M_{\overline{z^t \phi}}\;\;e_{-kq}) = \mathcal{V}_k^n(\overline{z^t \phi}) e_q  = M_\chi$ where $ \chi = \mathcal{V}_k^n(\overline{z^t \phi})$. Since, $ (\mathcal{S}_{\phi}^{k,n} M_{z^t})^*$ is compact then $ \mathcal{V}_k^n (\mathcal{S}_{\phi}^{k,n} M_{z^t})^*$  is also compact which implies $ M_\chi$  is compact. Thus, $\chi = 0 $. 

Finally, $ \big\langle \mathcal{V}_k^n(\overline{z^t \phi}), z^m \big\rangle = 0\Rightarrow \bigg\langle  z^{-t} \displaystyle \sum_{r \in \mathbb{Z}^n} \overline{a}_r z^{-r}, z^{-km}\bigg\rangle  = 0 $\\
$\Rightarrow \sum_{r \in \mathbb{Z}^n} \overline{a}_{-r-t} z^r, z^{-km} = 0 \Rightarrow \overline{a}_{km -t } = 0 \;\; \forall ~m \in \mathbb{Z}^n $ and $ t = 1, 2,\cdots ,k^2 - 1 $. So, $ a_r = 0 \;\; \forall ~r \in  \mathbb{Z}^n $ which implies $ \phi = 0$ and hence $ \mathcal{S}_{\phi}^{k,n} = 0 $.
\end{proof}

\begin{theorem}
$  \mathcal{S}_{\phi}^{k,n} \mathcal{S}_{\psi}^{k,n} =  \mathcal{S}_{\psi}^{k,n} \mathcal{S}_{\phi}^{k,n} $ if and only if $ \phi(z^{-k}) \psi(z) = \psi(z^{-k}) \phi(z)$ for $\phi, \psi \in L^\infty(\mathbb{T}^n)$.
\end{theorem}
\begin{proof} We have
\begin{flalign}\label{eqn7}
 \mathcal{S}_{\phi}^{k,n} \mathcal{S}_{\psi}^{k,n} -  \mathcal{S}_{\psi}^{k,n} \mathcal{S}_{\phi}^{k,n} = & \; \mathcal{V}_k^n M_\phi \mathcal{V}_k^n M_\psi - \mathcal{V}_k^n M_\psi \mathcal{V}_k^n M_\phi \cr = & \; \mathcal{V}_k^n \mathcal{V}_k^n \big( M_{\phi(z^{-k}) \psi(z)} - M_{\psi(z^{-k}) \phi(z)} \cr = & \;   \mathcal{S}_{1}^{k,n}  \mathcal{S}_{\phi(z^{-k}) \psi(z) - \psi(z^{-k}) \phi(z)}^{k,n}.
\end{flalign}
Now, $  \mathcal{S}_{1}^{k,n}  \mathcal{S}_{\phi(z^{-k}) \psi(z) - \psi(z^{-k}) \phi(z)}^{k,n} = 0  $ if and only if $ \phi(z^{-k}) \psi(z) = \psi(z^{-k}) \phi(z) $  by Theorem \ref{thm3}.
\end{proof}
 
\begin{theorem} \label{thm9}
For $\phi, \psi \in L^\infty(T)^n $, the product of two $k^{th}$ order slant Hankel  operator of level $n$  $ \mathcal{S}_{\phi}^{k,n}  \mathcal{S}_{\psi}^{k,n} $ is compact if and only if $ \phi(z^{-k}) \psi(z) = 0 .$
\end{theorem}
\begin{proof}
Let  $ \mathcal{S}_{\phi}^{k,n}  \mathcal{S}_{\psi}^{k,n} $ be compact, then both $ (\mathcal{S}_{\phi}^{k,n}  \mathcal{S}_{\psi}^{k,n})^* $ and $ (\mathcal{V}_k^n)^2 (\mathcal{S}_{\phi}^{k,n}  \mathcal{S}_{\psi}^{k,n})^* $ are compact. Now,
\begin{flalign}\label{eqn8}
(\mathcal{V}_k^n)^2 (\mathcal{S}_{\phi}^{k,n}  \mathcal{S}_{\psi}^{k,n})^* e_m(z) = & \; (\mathcal{V}_k^n)^2 M_{\overline{\psi}} (\mathcal{V}_k^n)^* M_{\overline{\phi}}  (\mathcal{V}_k^n)^* z^m \cr 
= & \; (\mathcal{V}_k^n)^2 M_{\overline{\psi} \overline{\phi}(z^{-k})} \big\lbrace(\mathcal{V}_k^n)^*\big\rbrace^2 z^m \cr 
= & \; (\mathcal{V}_k^n)^2 \big(\overline{\psi} \overline{\phi}(z^{-k}) z^{k^2m}\big) \cr
= & \;  \mathcal{V}_k^n \big\lbrace z^{-km} \mathcal{V}_k^n  \big( \overline{\psi} \overline{\phi}(z^{-k}) \big)\big\rbrace \cr
= & \; (\mathcal{V}_k^n)^2 \big( \overline{\psi} \overline{\phi}(z^{-k}) \big)z^m \cr
= & \; M_\Lambda e_m(z)
\end{flalign}
 where $ \Lambda = (\mathcal{V}_k^n)^2 \big( \overline{\psi} \overline{\phi}(z^{-k}) \big) $.  Since, $ (\mathcal{V}_k^n)^2 (\mathcal{S}_{\phi}^{k,n}  \mathcal{S}_{\psi}^{k,n})^* $ is compact then  $ \Lambda = 0 $.
 
 So, $ \big\langle  \Lambda,  e_m(z) \big\rangle =  \big\langle (\mathcal{V}_k^n)^2 \big( \overline{\psi} \overline{\phi}(z^{-k}) \big), z^m \big\rangle =  \big\langle \overline{\psi} \overline{\phi}(z^{-k}, z^{k^2m} \big\rangle  = \overline{a}_{k^2m} $ where ${\psi} {\phi}(z^{-k}) = \displaystyle \sum_{m \in \mathbb{Z}^n}a_m z^m $.  Hence, $  \overline{a}_{k^2m} = 0 \;\; \forall m \in \mathbb{Z}^n $.
 
 Again, let us consider $\mathcal{S}_{\phi}^{k,n}  \mathcal{S}_{\psi}^{k,n} M_{z^t},\;  t = 1, 2, \cdots, k^2-1. $ Since, $\mathcal{S}_{\phi}^{k,n}  \mathcal{S}_{\psi}^{k,n} $ is compact so, $\mathcal{S}_{\phi}^{k,n}  \mathcal{S}_{\psi}^{k,n} M_{z^t} $ is also compact for $ t = 1, 2, \cdots, k^2-1. $ We have,
\begin{eqnarray}
&& (\mathcal{V}_k^n)^2  (\mathcal{S}_{\phi}^{k,n}  \mathcal{S}_{\psi}^{k,n} M_{z^t})^* \nonumber\\
 &= &\;  (\mathcal{V}_k^n)^2 M_{\overline{z}^t} ( \mathcal{S}_{\psi}^{k,n})^* (\mathcal{S}_{\phi}^{k,n})^* \cr
& = &\;  (\mathcal{V}_k^n)^2 \overline{z}^t ( \mathcal{S}_{\psi}^{k,n})^* (\mathcal{S}_{\phi}^{k,n})^* \cr
& = &\; M_\Omega, \; {\rm where} \; \Omega = (\mathcal{V}_k^n)^2 \big(\overline{z}^t \overline{\psi} \overline{\phi}(z^{-k}) \big) [\rm \;by  \; equation \; (\ref{eqn8})].
\end{eqnarray}
 $ M_\Omega $ is compact implies $ \Omega = 0 $. Now, $ \big\langle  \Omega, e_m(z)\big\rangle = \big\langle (\mathcal{V}_k^n)^2 \big(\overline{z}^t \overline{\psi} \overline{\phi}(z^{-k}) \big), z^m \big\rangle = \big\langle  \overline{\psi} \overline{\phi}(z^{-k}), z^{k^2m  + t }\big\rangle = a_{-k^2m-t}$. Therefore, $ a_{-k^2m-t} = 0 $. From the above results, we get $ \psi(z) \phi(z^{-k})  = 0.$ \\[1mm] 
\noindent The proof of the converse part follows directly from Theorems \ref{thm4} and \ref{thm7}.
\end{proof}

\begin{theorem}
For any $ \phi \in L^\infty(\mathbb{T}^n)$, the $ k^{th}$ order slant Hankel operator of level $n$, $ \mathcal{S}_{\phi}^{k,n}$ satisfies the following properties
\begin{enumerate}
\item[(a)] $ \mathcal{S}_{1}^{k,n}  \mathcal{S}_{\phi}^{k,n} = 0 $  if and only if $ \phi = 0 $.
\item[(b)] $ \mathcal{S}_{1}^{k,n}  \mathcal{S}_{\phi}^{k,n}  $ is compact if and only if  $ \phi = 0 $.
\end{enumerate}
\end{theorem}

\begin{proof}
Let $\phi(z)= \displaystyle \sum_{m \in \mathbb{Z}^n} a_m z^m $. To prove $(a)$ let us assume that $ \mathcal{S}_{1}^{k,n}  \mathcal{S}_{\phi}^{k,n} = 0 $. Then, $ \mathcal{S}_{1}^{k,n}  \mathcal{S}_{\phi}^{k,n} = (\mathcal{V}_k^n)^2 M_\phi $. Now, $ \big\langle  \mathcal{S}_{1}^{k,n}  \mathcal{S}_{\phi}^{k,n} z^m, z^{m^\prime}\big\rangle =  \big\langle (\mathcal{V}_k^n)^2 M_\phi z^m, z^{m^\prime}  \big\rangle =  \big\langle   \mathcal{S}_{\phi}^{k,n} z^m, z^{-k m^\prime} \big\rangle  = a_{k^2m^\prime - m }$. So, $a_{k^2m^\prime - m } = 0 \;\;  \forall m, m^\prime \in \mathbb{Z}^n $ which implies $ a_m = 0  \;\;\forall m \in \mathbb{Z}^n $. Hence, $\phi(z)= 0 $. The proof of the converse part is obvious.\\[1mm]
\noindent The proof of the second part $ (b)$ immediately follows from Theorem \ref{thm9}.
\end{proof}
\begin{theorem}
If $ \phi, \psi \in L^\infty(\mathbb{T}^n)$ then, the given statements are equivalent
\begin{enumerate}
\item[(a)] $ \mathcal{S}_\phi^{k,n} $ and  $ \mathcal{S}_\psi^{k,n} $ commute.
\item[(b)]  $ \mathcal{S}_\phi^{k,n} $ and  $ \mathcal{S}_\psi^{k,n} $ essentially commute.
\item[(c)] $ \phi(e_{-k}(z)) \psi = \psi(e_{-k}(z))\phi $
\end{enumerate}
\end{theorem} 
\begin{theorem}
A $k^{th}$ order slant Hankel operator of level $n$,  $ \mathcal{S}_\phi^{k,n} $ cannot be isometric.
\end{theorem}
\begin{proof}
Let $\phi(z)= \displaystyle \sum_{r \in \mathbb{Z}^n} a_r z^r $. If possible, let us assume that $ \mathcal{S}_\phi^{k,n} $ is isometric. Then, $ \big\Vert \mathcal{S}_\phi^{k,n}  e_m(z)\big \Vert = \big\Vert e_m(z)\big\Vert  = 1 \;\;\forall m \in \mathbb{Z}^n $ which implies $ \bigg\Vert  \displaystyle \sum_{r \in \mathbb{Z}^n} a_{-kr -m } z^r \bigg\Vert^2 = 1$  that is $   \displaystyle \sum_{r \in \mathbb{Z}^n} \vert a_{-kr -m }\vert^2 = 1 \;\; \forall m, r \in \mathbb{Z}^n $. We have, 
\begin{flalign}\label{eqn10}
\mathcal{S}_\phi^{k,n} (\mathcal{S}_\phi^{k,n})^* e_p(z) = & \; \mathcal{V}_k^n M_\phi M_{\overline{\phi}}(\mathcal{V}_k^n)^* e_p(z) \cr
= & \; \mathcal{V}_k^n M_\phi M_{\overline{\phi}} \; e_{-kp}(z)  \cr = & \; \mathcal{V}_k^n \bigg( \displaystyle \sum_{r \in \mathbb{Z}^n} a_r z^r \sum_{r \in \mathbb{Z}^n} \overline{a}_r z^{-r} e_{-kp}(z)\bigg)   \cr 
= & \;\mathcal{V}_k^n  \bigg(\sum_{r \in \mathbb{Z}^n} \vert a_r \vert^2 \bigg) \; e_p(z) \cr
= & \; M_\Upsilon e_p(z), \; {\rm where} \; \Upsilon = \mathcal{V}_k^n(\vert \phi \vert^2).  
\end{flalign}
Now, $ \Vert \phi \Vert^2 = \bigg\Vert \displaystyle \sum_{r \in \mathbb{Z}^n} a_r z^r \bigg\Vert^2 = \sum_{r \in \mathbb{Z}^n} \vert a_r \vert^2  $\\
$= \displaystyle\sum_{i_j = 0, 1 \leq j \leq n} ^ {k-1}~~ \displaystyle\sum_{(r_1,r_2, \cdots r_n)\in \mathbb{Z}^n} \vert a_{kr_1 - i_1,k r_2 - i_2, \cdots, kr_n - i_n } \vert  = k^n $. So, 
\begin{equation}\label{eqn11}
\big\Vert \vert \phi\vert^2 \big\Vert = \big\Vert \mathcal{S}_\phi^{k,n} (\overline{\phi}) \big\Vert = \big\Vert \phi\big\Vert = k^{\frac{n}{2}}.
\end{equation}
As we know that $\big\Vert  \mathcal{S}_\phi^{k,n}  \big\Vert = \big\Vert \mathcal{V}_k^n \big\Vert_\infty ^{\frac{1}{2}}$ and by assumption $ \mathcal{S}_\phi^{k,n}  $ is an isometry. Therefore, $ \big\Vert \mathcal{V}_k^n(\vert \phi \vert^2)\big\Vert_\infty = 1 $. Since, $ k \geq 2 $ and $ n \geq 1 $, the above calculations justify that $ \big\Vert \mathcal{V}_k^n(\vert \phi \vert^2)\big\Vert_\infty \leq \mathcal{V}_k^n(\vert \phi \vert^2)\big\Vert_2 \big\Vert $, which is a contradiction. Hence, $ \mathcal{S}_\phi^{k,n}  $ cannot be isometric.
\end{proof}

\end{document}